\documentclass[12pt,reqno]{amsart}
\DeclareUnicodeCharacter{3002}{.}
\usepackage[nolonger]{optional}
\usepackage{amssymb,bbm,graphicx}
\usepackage{color}
\usepackage{comment}
\usepackage{enumitem}
\usepackage[colorlinks=true,linkcolor=black,citecolor=black,urlcolor=purple]{hyperref} 
\usepackage{flushend}
\usepackage{tikz}
\usepackage{amsmath}
\usepackage{mathtools}

\usepackage{cite}

\DeclareMathSymbol{\shortminus}{\mathbin}{AMSa}{"39} % short minus
\usepackage{tikz-3dplot}
\usepackage{tikz-cd}
\usepackage{xcolor}
\usetikzlibrary{arrows,decorations,decorations.markings,calc,positioning,shapes.geometric,arrows.meta} 
\usepackage[all]{xy}
\usepackage[margin=30mm]{geometry}
\usepackage{mathrsfs}
\usepackage{appendix}

\usepackage{setspace}
\newcommand{\marginparstretch}{0.6}
\let\oldmarginpar\marginpar
\renewcommand\marginpar[1]{\-\oldmarginpar[\framebox{\setstretch{\marginparstretch}\begin{minipage}{\marginparwidth}{\raggedleft\tiny #1}\end{minipage}}]{\framebox{\setstretch{\marginparstretch}\begin{minipage}{\marginparwidth}{\raggedright\tiny #1}\end{minipage}}}}

\numberwithin{equation}{section}
\numberwithin{figure}{section}
\theoremstyle{plain}
\newtheorem{theorem}{Theorem}[section]
\newtheorem*{keythmnonumber}{Theorem}
\newtheorem{proposition}[theorem]{Proposition}
\newtheorem{lemma}[theorem]{Lemma}

\theoremstyle{remark}
\newtheorem{remark}[theorem]{Remark}
\theoremstyle{remark}

\theoremstyle{definition}
\newtheorem{definition}[theorem]{Definition}

\allowdisplaybreaks[3]

\newcommand\define[1]{#1}

\newcommand\cO{\mathcal{O}}

\DeclareMathOperator\Hom{Hom}
\DeclareMathOperator\dHom{Hom}
\newcommand\dHomk{\dHom^\bullet}

\providecommand{\cE}{\mathcal{E}}
\providecommand{\cF}{\mathcal{F}}

\providecommand{\cL}{\mathcal{L}}

\providecommand{\R}{\mathbb{R}}
\providecommand{\PP}{\mathbb{P}}
\providecommand{\Z}{\mathbb{Z}}

\providecommand{\C}{\mathbb{C}}

\providecommand{\F}{\mathbb{F}}
\providecommand{\Br}{\mathrm{Br}}
\providecommand{\Aut}{\mathrm{Aut}}

\providecommand{\AHilb}{\operatorname{\mathit{A}\text{-}Hilb}}
\providecommand{\SL}{\operatorname{SL}}
\providecommand{\diag}{\operatorname{diag}}

\newcommand\sm\shortminus

\newcommand\newrays{\tau}

\title[Quiver braid group action for a 3-fold crepant resolution]{Quiver braid group action \\ for a 3-fold crepant resolution}
\author{Will Donovan}
\address[Will Donovan]{Yau MSC, Tsinghua University, Haidian, Beijing, China; BIMSA, Yanqi Lake, Huairou, Beijing, China; Kavli IPMU (WPI), TODIAS, University of Tokyo, Kashiwanoha, Chiba, Japan}
\email{donovan@mail.tsinghua.edu.cn}
\author{Luyu Zheng}
\address[Luyu Zheng]{Yau MSC, Tsinghua University, Haidian, Beijing, China}
\email{zhengluyu84@gmail.com}
\thanks{The authors are supported by Yau MSC, Tsinghua University, and China TTP. The first author is supported by Yanqi Lake BIMSA}
\begin{document}

%\phantom{.}\vspace{-1.5cm} 

\thispagestyle{empty}

\begin{abstract}
The 3-fold cyclic quotient singularity denoted $\tfrac{1}{7}(1,2,4)$ admits a crepant resolution $X$ with three exceptional Hirzebruch surfaces intersecting pairwise along curves. We show that the derived category $D(X)$ carries a faithful action of a quiver braid group, where the relevant quiver is a 3-cycle encoding the intersection data.
\end{abstract}
\maketitle

% 14F08 Derived categories of sheaves, dg categories, and related constructions in algebraic geometry
% 14J32 Calabi-Yau manifolds, mirror symmetry
% 18G80 Derived categories, triangulated categories

%\newpage
\setcounter{tocdepth}{1}
\tableofcontents
%\newpage

\section{Introduction}

The minimal resolutions~$M$ of Du Val surface singularities are key objects in the intersection of algebra, geometry, and theoretical physics. They admit an ADE classification and arise, for instance, as Slodowy slices, hyper-Kähler quotients, and asymptotically locally Euclidean (ALE) spaces. In particular, they are classic examples in mirror symmetry. 

The derived categories of coherent sheaves~$D(M)$ carry faithful Artin braid group actions~\cite{ST}, which may be thought of as mirror to symplectic braid group actions. In these actions, the braid group generators act by twist autoequivalences associated with the components of the exceptional locus. 
 
It is then natural to seek three-dimensional generalizations. A large class is provided by crepant resolutions of 3-fold cyclic quotient singularities, whose rich toric combinatorics makes their derived categories amenable to explicit study. However, the appropriate generalization of the above Artin braid group actions to this setting remains to be understood.

In this paper, we construct a faithful action of a \emph{quiver} braid group~\cite{GM,Qiu} in the above 3-fold setting. We work with a concrete example~$X$ given in Section~\ref{sec.main}, in which the exceptional locus of the resolution consists of three components~${S}_l$ each isomorphic to the Hirzebruch surface~$\mathbb{F}_2$. These naturally form a cycle in the sense that the intersection of ${S}_l$ and ${S}_{l+1}$ is a section of ${S}_l$ and a fibre of ${S}_{l+1}$, as  in Figure~\ref{fig.rel2}.

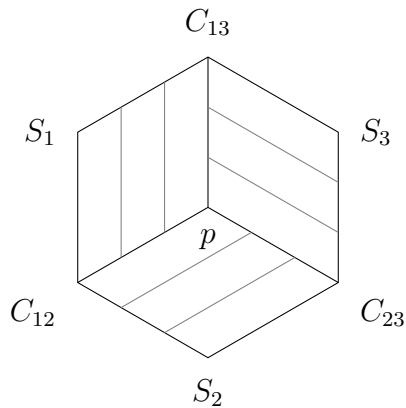
\begin{figure}[h!]
    \centering
    \begin{tikzpicture}[scale=2,semithick,
dot/.style = {circle, fill=black, inner sep=1pt}]
        \coordinate (O) at (0,0);
        \coordinate (V2) at (90:1); 
\coordinate (V3) at (210:1); 
\coordinate (V1) at (330:1);
\coordinate (S2) at (-90:1); 
\coordinate (S1) at (150:1); 
\coordinate (S3) at (30:1);
\coordinate (V11) at ($ (V2) !1/3! (S1) $);
\coordinate (V12) at ($ (O) !1/3! (V3) $);
\coordinate (V11b) at ($ (V2) !2/3! (S1) $);
\coordinate (V12b) at ($ (O) !2/3! (V3) $);
\coordinate (V21) at ($ (V3) !1/3! (S2) $);
\coordinate (V22) at ($ (O) !1/3! (V1) $);
\coordinate (V21b) at ($ (V3) !2/3! (S2) $);
\coordinate (V22b) at ($ (O) !2/3! (V1) $);
\coordinate (V31) at ($ (V1) !1/3! (S3) $);
\coordinate (V32) at ($ (O) !1/3! (V2) $);
\coordinate (V31b) at ($ (V1) !2/3! (S3) $);
\coordinate (V32b) at ($ (O) !2/3! (V2) $);

\draw (S3) -- (V2) -- (S1) -- (V3) -- (S2) -- (V1) -- cycle; 
\draw (V1) -- (O);
\draw (V2) -- (O);
\draw (V3) -- (O);
\draw[black!50] (V11) -- (V12);
\draw[black!50] (V11b) -- (V12b);
\draw[black!50] (V21) -- (V22);
\draw[black!50] (V21b) -- (V22b);
\draw[black!50] (V31) -- (V32);
\draw[black!50] (V31b) -- (V32b);
 
\node[label={180:${S}_1$}] at (S1) {};
\node[label={270:${S}_2$}] at (S2) {};
\node[label={0:${S}_3$}] at (S3) {};
\node[dot,label={270:$p$}] at (O) {};
    \end{tikzpicture}
    \caption{The three exceptional surfaces in $X$}
    \label{fig.rel2}
\end{figure}

This configuration suggests the quiver shown in Figure~\ref{fig.quiverW}, with a superpotential~$W$ corresponding to the cycle. We show that the associated quiver braid group~$G$, given below, acts faithfully on the derived category $D(X)$.

\begin{figure}[h!]
\begin{center}
\begin{tikzcd}
	& 1 \\
	2 && 3
	\arrow[from=1-2, to=2-1,"a"']
	\arrow[from=2-1, to=2-3,"b"']
	\arrow[from=2-3, to=1-2,"c"']
\end{tikzcd}
\end{center}

\caption{Quiver $(Q,W=cba)$}

\label{fig.quiverW}
\end{figure}
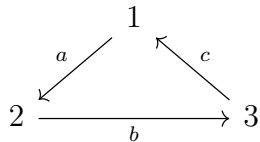

\subsection{Existing work} 

Around 2000, Seidel and R.~P.~Thomas~\cite{ST} constructed a faithful action of the Artin braid group $\Br_{n+1}$ on $D(X)$ from an $A_n$-configuration of spherical objects in $D(X)$. We briefly recall this theory in Section~\ref{sec.Generalities}.  Subsequently, Brav and H.~Thomas~\cite{BT} established faithfulness for configurations of $2$-spherical objects in all ADE types. Their argument was later simplified by Hirano and Wemyss~\cite{HW} using methods involving hyperplane arrangements.

Qiu and Woolf~\cite{QW} gave a faithful action of a certain braid group on $D(\Gamma_N\, Q)$ for $\Gamma_N\, Q$ the Calabi--Yau-$N$ Ginzburg algebra associated with a Dynkin quiver $Q$. Independently, Nordskova and Volkov~\cite{NV} proved general faithfulness results in an enhanced triangulated category setting for configurations in all simply-laced Dynkin types.

Seidel and Thomas give examples of $3$-folds $X$ containing configurations of surfaces which yield $A_n$-configurations and thence braid group actions~\cite[end of Section~3]{ST}. The first author and Wemyss showed that, for a $3$-fold $X$ with Gorenstein terminal singularities and individually floppable irreducible exceptional curves, the fundamental group of the complexified complement of a real hyperplane arrangement acts on $D(X)$ via flop functors~\cite[Theorem~1.2]{DW}. The pure braid groups, namely the kernels of the natural projections $\Br_n\to \mathfrak{S}_n$, occur as examples of these fundamental groups.

\subsection{Main theorem}\label{sec.main}

Take $X=\AHilb(\C^3)$ a crepant resolution of~$\C^3/A$, as recalled in Section~\ref{AHilb}, with $A=\mu_7$ acting with weights~$(1,2,4)$. The quiver braid group~\cite{GM,Qiu} associated with the quiver with potential $(Q,W)$ in Figure~\ref{fig.quiverW} is
\[
G=\langle g_1,g_2,g_3 \mid g_1g_2g_3g_1=g_2g_3g_1g_2,\;\; g_ig_jg_i=g_jg_ig_j \rangle,
\]
also denoted $\mathrm{AT}(Q,W)$ and called the algebraic twist group by Qiu.

We show that the spherical twists associated with the exceptional surfaces in $X$ realize this group:

\begin{keythmnonumber}[Theorem~\ref{thm.main}]
    The spherical twists $T_l$ associated with the  three exceptional surfaces $S_l$ in $X$ generate a subgroup of $\mathrm{Aut}\,D(X)$ isomorphic to $G$ above.
\end{keythmnonumber}

We sketch the idea of the proof. By the toric analysis of Reid and Craw~\cite{CR,Craw}, the junior simplex of $X$ is shown in Figure~\ref{fig.js7cycle}. The fan of $X$ is the cone over this simplex, and the three highlighted vertices correspond to exceptional surfaces~${S}_l$ each isomorphic to the Hirzebruch surface~$\mathbb{F}_2$.

\begin{figure}[h!]
\begin{center}
\begin{tikzpicture}[scale=3,semithick,
dot/.style = {circle, fill=black, inner sep=1.5pt}]

\coordinate (V1) at (90:1); 
\coordinate (V2) at (210:1); 
\coordinate (V3) at (330:1); 

\coordinate (V12mid) at ($ (V1) !1/3! (V2) $);
\coordinate (V23mid) at ($ (V2) !1/3! (V3) $);
\coordinate (V31mid) at ($ (V3) !1/3! (V1) $);

\coordinate (P1) at ($ (V12mid) !1/7! (V3) $);
\coordinate (P2) at ($ (V23mid) !1/7! (V1) $);
\coordinate (P3) at ($ (V31mid) !1/7! (V2) $);

\draw (V1) -- (V2) -- (V3) -- cycle; 

\draw (V3) -- (P1); 
\draw (V1) -- (P2); 
\draw (V2) -- (P3);

\draw (V2) -- (P1); 
\draw (V3) -- (P2); 
\draw (V1) -- (P3);

\node[dot,label={180:${S}_1$}] at (P1) {};
\node[dot,label={270:${S}_2$}] at (P2) {};
\node[dot,label={270:${S}_3$}] at (P3) {};

\end{tikzpicture}
\end{center}
\caption{Junior simplex of $X$}
\label{fig.js7cycle}
\end{figure}
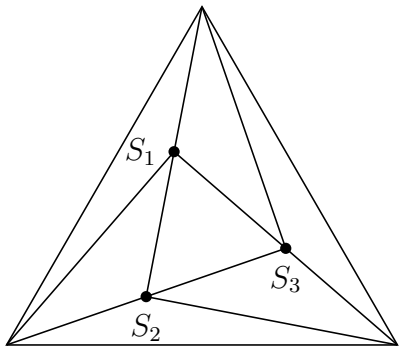

The structure sheaves of the ${S}_l$ induce spherical objects $\mathcal{E}_l$ on $X$; we write $T_l$ for the associated twist autoequivalences on $D(X)$. The $\{\mathcal{E}_l,\mathcal{E}_m\}$ form $A_2$-configurations, hence the degree~$3$ braid relations follow from~\cite{ST}. The remaining relation requires a new strategy. We first describe $T_1\mathcal{E}_2$ as a pushforward of a line bundle on ${S}_1\cup{S}_2$. Using this, we find that  $T_1\mathcal{E}_2$ is orthogonal to $\cE_3$. This implies that $T_3(T_1\mathcal{E}_2)\cong T_1\mathcal{E}_2$, and by general facts about spherical twists the relation $T_1T_2T_3T_1 = T_2T_3T_1T_2$ follows, giving a $G$-action on $D(X)$. Finally, setting $g_3' = g_2 g_3 g_2^{-1}$ yields an isomorphism $G \cong \Br_4$, and faithfulness of the $G$-action may then be deduced from the faithfulness result of~\cite{ST}.

\subsection{Contents}
In Section~\ref{sec.sph}, we review spherical objects in $D(X)$, their associated twists, $A_n$-configurations, and braid group actions; we establish structural results for constructing new configurations via spherical twists on Calabi–Yau varieties and compute $\Hom$ spaces associated with Cartier divisors.
In Section~\ref{sec.toric}, we briefly recall toric varieties via fans and orbits, construct toric line bundles from Cartier divisors, and discuss examples including Hirzebruch surfaces and $A$-Hilbert schemes, describing the crepant resolution and compact exceptional surfaces. 
In Section~\ref{sec.cyc}, we explain the geometry of the three irreducible compact exceptional surfaces in $X$. In Section~\ref{sec.proof}, we prove our main theorem.

\subsection{Conventions}

We work over $\C$, write $D(X)$ for the bounded derived category of coherent sheaves on a variety, and take functors to be  derived.

\subsection*{Acknowledgements}
We would like to thank Nick Addington, Alastair Craw and Yukari Ito for useful discussions, Shan Peng for conversations on affine braid groups, and Qiu Yu for suggestions on algebraic twist groups and quiver mutations.

\section{Spherical objects}\label{sec.sph}

\subsection{Generalities}\label{sec.Generalities} We recall the theory of spherical twists from~\cite{ST}, for $X$ smooth of dimension~$d$. Here for simplicity we assume $X$ projective, though our application in Section~\ref{sec.proof} is to $X$ quasi-projective: we note there why results extend.

\begin{definition}
An object $\mathcal{E}\in D(X)$ is \emph{spherical} if 
\begin{itemize}
    \item 
    $\dHomk(\mathcal{E},\mathcal{E})=\mathbb{C}\oplus\mathbb{C}[-d]$
    and 
    \item $\mathcal{E}\otimes\omega_X\cong\mathcal{E}$
\end{itemize}
where $\dHomk(\mathcal{A},\mathcal{B})$ is the graded vector space $\bigoplus_r \Hom_{D(X)}(\mathcal{A},\mathcal{B}[r])$.\end{definition}

For any $\mathcal{F} \in D(X)$, one can define the \emph{spherical twist} $T_{\mathcal{F}}\colon D(X)\to D(X)$ as in~\cite[Definition~8.3]{Huy}. Moreover, for any spherical object $\mathcal{E}$, the action of the spherical twist on objects can be expressed as
\begin{equation}\label{eq.cone}
T_{\mathcal{E}}(\mathcal{A}) \cong
\operatorname{Cone}\big(
    \mathcal{E} \otimes \dHomk(\mathcal{E}, \mathcal{A})
    \longrightarrow \mathcal{A}
\big),
\end{equation}
where the morphism is the natural evaluation map~\cite[Exercise~8.5]{Huy}.
If $\mathcal{E}$ is spherical, then $T_{\mathcal{E}}$ is an exact autoequivalence of $D(X)$, see for instance~\cite[Proposition~8.6]{Huy}.
For two spherical objects $\mathcal{E},\mathcal{F}\in D(X)$ with $\dim_{\mathbb{C}}\dHomk(\mathcal{E},\mathcal{F})=h$, the following hold:
\begin{itemize}
    \item If $h=0$, then $T_{\mathcal{E}}T_{\mathcal{F}}\cong T_{\mathcal{F}}T_{\mathcal{E}}$.
    \item If $h=1$, then  $T_{\mathcal{E}}T_{\mathcal{F}}T_{\mathcal{E}}
      \cong T_{\mathcal{F}}T_{\mathcal{E}}T_{\mathcal{F}}$ (the braid relation).
    \item If $h\ge2$, then  
    $\langle T_{\mathcal{E}},T_{\mathcal{F}}\rangle$ forms a subgroup of $\Aut\, D(X)$~\cite[Theorem~1.1]{Keat}.
\end{itemize}

\begin{definition}
An \emph{$A_n$-configuration} in $D(X)$ is a collection of spherical objects 
$\mathcal{E}_1,\dots,\mathcal{E}_n$ such that
\[
 \dim_{\mathbb{C}}\dHomk(\mathcal{E}_i,\mathcal{E}_j)=
 \begin{cases}
   1, & |i-j|=1,\\
   0, & |i-j|\ge 2.
 \end{cases}
\]
\end{definition}
For $d\geq 2$, such a configuration induces a faithful action of the braid group 
$\Br_{n+1}$ on $D(X)$ via spherical twists $T_i\coloneqq T_{\cE_i}$~\cite[Theorem~1.3]{ST}.

\begin{lemma}\label{lemma.exten}
    Take $\{\cE_1,\cE_2\}$ an $A_2$-configuration with $\dHom^r(\cE_1,\cE_2)=\C$.
    \begin{itemize}
        \item[(1)] $T_2^{-1}\cE_1\cong T_1\cE_2[r-1]$
        \item[(2)] If $r=1$, ${T_1}\cE_2$ is determined by a non-split exact triangle $\cE_2\rightarrow T_1\cE_2\rightarrow\cE_1\rightarrow$.
    \end{itemize}
\end{lemma}
\begin{proof} For (1), see~\cite[proof of Proposition~2.13]{ST}. For (2), the non-split exact triangle is obtained using~\eqref{eq.cone}, and it determines $T_1\cE_2$ since $\Hom_{D(X)}(\cE_1,\cE_2[1])=\C$.  
\end{proof}

\begin{proposition}\label{prop.a23}
    Assume that $X$ is Calabi--Yau, so that $\omega_{X}$ is trivial.
    \begin{itemize}
        \item[(1)] If $\{\cE,\cF\}$ is an $A_2$-configuration, then $\{\cE,T_{\cE}\cF\}$ is also.
        \item[(2)] If $\{\mathcal{E}_l,\mathcal{E}_m\}$ are $A_2$-configurations for each pair $\{l,m\}\subset\{1,2,3\}$, and furthermore  $T_1\cE_2\in\cE_3^\perp$, then $\{\cE_1,\cE_2,T_2\cE_3\}$ is an $A_3$-configuration.
            \end{itemize}
\end{proposition}
\begin{proof}
By Calabi--Yau Serre duality, it suffices to check the $A_n$-configuration condition for $i>j$. For (1), $T_{\cE}\cF$ is  spherical using that $T_{\cE}$ is an equivalence, and that $X$ is Calabi--Yau. Then using $T_\cE\cE\cong \cE[1-d]$, we obtain
    \begin{equation*}
        \dHomk(T_{\cE}\cF,\cE)\cong\dHomk(\cF,T_{\cE}^{-1}\cE)\cong  \dHomk(\cF,\cE)[d-1].
    \end{equation*}
    For (2), we have 
    \begin{equation*}
\dHomk(T_2\cE_3,\cE_1)\cong \dHomk(\cE_3,T_2^{-1}\cE_1)=0
    \end{equation*}
using Lemma~\ref{lemma.exten}(1), and applying (1) to $\{\mathcal{E}_2,\mathcal{E}_3\}$ then gives the result.
\end{proof}

\subsection{Calculations}

The following results will help us calculate $\Hom$ spaces between spherical objects later. Take two Cartier divisors $D_1$ and $D_2$ in ${X}$ smooth quasi-projective with embeddings $i_l:D_l\to {X}$. Assume their intersection $C\coloneqq D_1\cap D_2$ is also Cartier in each $D_l$, with embeddings $j_l\colon C\to D_l$ and $k\colon C\to X$.
\begin{equation}\label{fig.inters}
    \begin{tikzcd}
                   & {X}                                   &                    \\
{D_1} \arrow[ru, hook,"i_1"] &                                     & {D_2} \arrow[lu, hook', "i_2"'] \\
                   & C \arrow[lu, hook, "j_1"] \arrow[ru, hook', "j_2"'] \arrow[uu, hook, "k"] &                   
\end{tikzcd}
\end{equation}

\begin{proposition}\label{prop.extgen}
    Let $\cF$ be an object of $D(D_2)$. Then
\begin{equation*}
    \mathcal{H}om_{X}(i_{1*}\mathcal{O}_{D_1}, i_{2*}\mathcal{F})\cong k_{*}( \cO_{C}(D_1)\otimes j_{2}^*\mathcal{F})[-1].
\end{equation*}
\end{proposition}
\begin{proof}
We have
\begin{equation*}
    \mathcal{H}om_{X}( i_{1*}\mathcal{O}_{D_1},i_{2*}\cF)
    \cong i_{2*}\mathcal{H}om_{D_2}({i_2^*} i_{1*}\mathcal{O}_{D_1},\cF)\cong i_{2*}(({i_2^*} i_{1*}\mathcal{O}_{D_1})^\vee\otimes\cF).
\end{equation*}
Note that $i_1\colon  D_1\to {X}$ is a closed immersion hence proper. By \cite[Proposition~A.1]{Add}, we then have base change for the square~\eqref{fig.inters}, so that
${i_2^*} i_{1*}\mathcal{O}_{D_1}\cong  j_{2*}j_1^*\mathcal{O}_{D_1}=  j_{2*}\mathcal{O}_{C}$.

To complete the argument, we claim $( j_{2*}\mathcal{O}_{C})^\vee\cong  j_{2*}\cO_{C}(D_1)[-1]$ and apply the projection formula.  This isomorphism is well known for $D_2$ and $C$ smooth, for instance using~\cite[Corollary~3.40]{Huy}. However smoothness is not needed, as follows.\footnote{We will apply this proposition in a singular case in the proof of Proposition~\ref{prop.ext1}.}
    
    Restricting the exact sequence $0\rightarrow\mathcal{O}_X(D_1)^\vee\rightarrow\mathcal{O}_X\rightarrow i_{1*}\mathcal{O}_{D_1}\rightarrow0$ to $D_2$ and using that $C=D_1\cap{D_2}$ gives  
    \begin{equation}\label{eq.res_on_D2}
        0\longrightarrow\mathcal{O}_{D_2}(D_1)^\vee\overset{s^\vee}{\longrightarrow}\mathcal{O}_{D_2}\longrightarrow  j_{2*}\mathcal{O}_{C}\longrightarrow 0
    \end{equation}
exact, and tensoring by $\mathcal{O}_{D_2}(D_1)$ yields the top line of a commutative diagram:
    \begin{equation*}
        \begin{tikzcd}
            0&{\mathcal{O}_{D_2}(D_1)^\vee\otimes\mathcal{O}_{D_2}(D_1)}&{\mathcal{O}_{D_2}(D_1)}&{ j_{2*}\mathcal{O}_{C}\otimes\mathcal{O}_{D_2}(D_1)}&0\\
            &{\mathcal{O}_{D_2}}&{\mathcal{O}_{D_2}(D_1)}&{ j_{2*}\mathcal{O}_{C}(D_1)}&
            \arrow[from=1-1, to=1-2]
            \arrow["s^\vee\otimes1", from=1-2, to=1-3]
            \arrow[from=1-3, to=1-4]
            \arrow[from=1-4, to=1-5]
            \arrow["s", from=2-2, to=2-3]
            \arrow[from=2-3, to=2-4]
            \arrow["\mathrm{ev}"', "\sim"{sloped}, from=1-2, to=2-2]
            \arrow["\sim"{sloped}, from=1-4, to=2-4]
            \arrow[equal, from=1-3, to=2-3]
        \end{tikzcd}
    \end{equation*}
    Hence $ j_{2*}\mathcal{O}_{C}(D_1)\cong \operatorname{Cone}(s)$. On the other hand $( j_{2*}\mathcal{O}_{C})^\vee\cong \operatorname{Cone}(s)[-1]$   using~\eqref{eq.res_on_D2}, and the claim is proved.
\end{proof}

\begin{lemma}\label{lemma.restocurve}
    If $C\cong \mathbb{P}^1$ has self-intersection number $s$ in $D_2$ then $\mathcal{O}_{C}(D_1)\cong\mathcal{O}_{C}(s).$
\end{lemma}
\begin{proof}
    We have $\mathcal{O}_{C}(D_1)=\mathcal{O}_X(D_1)|_{C}\cong \mathcal{O}_{D_2}(C)|_{C}\cong \mathcal{O}_{C}(s)$, where
the first isomorphism follows from $\mathcal{O}_{X}(D_1)|_{D_2}\cong \mathcal{O}_{D_2}(C)$, using that $D_1$ is Cartier in $X$ and  $C$ is  Cartier in  $D_2$. 
\end{proof}

\section{Toric geometry}\label{sec.toric}
\subsection{Generalities}

    A \define{toric variety} $X$ over $\C$ is a $d$-dimensional irreducible variety containing a dense open torus $T_N\cong (\C^*)^d$ such that the multiplication action on $T_N$ extends to an algebraic action $T_N\times X\to X$~\cite[Definition~3.1.1]{CLS}. Let $M=\Hom(T_N,\C^*)$ be the abelian group of characters of $T_N$ and $N=M^\vee$.\par 
    
    A \define{fan}~$\Sigma$ in $N_{\R}\coloneqq N\otimes_{\Z}\R$ is a finite collection of strongly convex rational polyhedral cones $\sigma\subset N_{\R}$ satisfying the usual compatibility conditions~\cite[Definition~3.1.2]{CLS}. We denote by $\Sigma(r)$ the set of $r$-dimensional cones in $\Sigma$, and write $\tau\prec \sigma$ if $\tau$ is a face of $\sigma$.\par 
    
    For any cone $\sigma$ in a fan $\Sigma$, the associated affine variety is 
    \[U_\sigma = \operatorname{Spec}\!\big(\C[\sigma^\vee \cap M]\big),
\qquad 
\sigma^\vee = \{\, m\in M_\R \mid \langle m,n\rangle \ge 0 
\text{ for all } n\in\sigma \,\}.
\]
If $\tau=\sigma_1\cap\sigma_2$ is a common face, then 
$U_\tau = U_{\sigma_1}\cap U_{\sigma_2}$.
Gluing these affine pieces gives the toric variety $X_\Sigma$ 
associated with~$\Sigma$~\cite[after Definition~3.1.2]{CLS}. There is an \define{orbit-cone correspondence} \cite[Theorem~3.2.6]{CLS} between cones and $T_N$-orbits in $X_{\Sigma}$ as follows.
\[
\{\sigma\in\Sigma\}
\;\Longleftrightarrow\;
\{\text{$T_N$-orbits in } X_\Sigma\}
\qquad
\sigma \longmapsto O(\sigma)\cong\Hom_\Z(\sigma^\perp\cap M,\C^*)
\]
Moreover, $\dim O(\sigma)=d-\dim\sigma$.  
Write $V(\sigma)=\overline{O(\sigma)}$.  
Then:
\begin{equation}\label{eq.orbitclosure}
V(\sigma)=\bigcup_{\sigma\preceq\tau} O(\tau)
\end{equation}

\begin{lemma}[\hspace*{-0.3em}{\cite[Lemmas~3.2.4 and~3.2.5]{CLS}}]\label{lemma.orbit}
    Let $N_\sigma$ be the sublattice of $N$ spanned by the points in $\sigma\cap N$ and let $N(\sigma)=N/N_\sigma$. Then $O(\sigma)\cong N(\sigma)\otimes_\Z \C^*$.
\end{lemma}

For any ray $\rho\in \Sigma(1)$, write $D_{\rho}\in \mathbf{Pic}(X_\Sigma)$ for the divisor corresponding to $V(\rho)$. To indicate $\rho=\langle n \rangle$ we often simply write the generator $n$ when no confusion seems likely.

\subsection{Line bundles}\label{sect.linebundles}

For a Cartier divisor $D=\sum_{\rho\in\Sigma(1)} a_\rho D_\rho$
on $X_\Sigma$, we construct a new fan 
$\Sigma\times D$ in $N_\R\times\R$ such that $X_{\Sigma\times D}$ is isomorphic to the total space of the line bundle $\mathcal{O}_{X_\Sigma}(D)$ with bundle projection induced by projection away from $\R$~\cite[Chapter~7.3]{CLS}. Namely, for each cone $\sigma\in\Sigma$, let
\begin{equation*}
    \tilde\sigma
\coloneqq \big\langle (\underline{0},1),\,(u_\rho,-a_\rho)\big| \rho\preceq\sigma,\,\rho\in\Sigma(1) \big\rangle
\end{equation*}
where $u_\rho$ denotes the primitive generator of $\rho$. Then $\Sigma\times D$ is given by the $\tilde\sigma$ for all $\sigma\in\Sigma$, along with all their faces. 

\begin{proposition}\label{prop.orbit}
The orbit $O(\sigma)\subset X_\Sigma$ is naturally identified with 
$O(\tilde{\sigma})\subset X_{\Sigma\times D}$.
\end{proposition}

\begin{proof}
Following the proof of~\cite[Proposition~7.3.1]{CLS}, set
\[\hat{\sigma}\coloneqq\big\langle (u_\rho,-a_\rho)\big| \rho\preceq\sigma,\,\rho\in\Sigma(1)  \big\rangle\]
and let $\hat{\Sigma}$ be given by the $\hat\sigma$ for all $\sigma\in\Sigma$, yielding a subfan of $\Sigma\times D$. Every cone in $\Sigma\times D$ decomposes uniquely as a Minkowski sum $\hat{\sigma}+(\underline{0},1)$ or $\hat{\sigma}+(\underline{0},0)$ with $\hat{\sigma}\in\hat{\Sigma}$, and projection away from $\R$ induces a bijection $\hat{\Sigma}\to\Sigma$. Letting then $\Sigma_0$ be the fan in $\{\underline{0}\}\times\R$ with the unique maximal cone $(\underline{0},1)$, we have $\Sigma\times D \cong \Sigma \oplus \Sigma_0$
in $N_{\Sigma\times D}\cong N_\Sigma\oplus N_{\Sigma_0}$.

Now $\tilde\sigma$ contains 
$N_{\Sigma_0}$ by construction, and projection away from $\R$ maps $\tilde\sigma$ onto $\sigma$. Hence $(N_{\Sigma\times D})_{\tilde\sigma}
\cong (N_\Sigma)_\sigma\oplus N_{\Sigma_0}$ giving
$N_{\Sigma\times D}(\tilde\sigma)
\cong N_\Sigma(\sigma)$ which, by Lemma~\ref{lemma.orbit}, induces a canonical isomorphism $O(\tilde\sigma)\cong O(\sigma)$. 
\end{proof}

\begin{proposition}\label{prop.divisor}
The subvariety $V((\underline{0},1))$
is the zero section of $X_{\Sigma\times D}$. Writing $i$ for its embedding, $V(\tilde\sigma)=i\,V(\sigma)$ for any $\sigma\in\Sigma$.
\end{proposition}

\begin{proof}
The first claim follows as the ray 
$(\underline{0},1)$
corresponds, by construction, to the fibre coordinate on 
$\mathcal{O}_{X_\Sigma}(D)$, as in the proof of~\cite[Proposition~7.3.1]{CLS}. The identification of Proposition~\ref{prop.orbit} equates disjoint unions
\[
\bigcup_{\tau\in \Sigma} O(\tilde\tau)=\bigcup_{\tau\in \Sigma} i\,O(\tau)
\]
and so \eqref{eq.orbitclosure} allows us to conclude.
\end{proof}

\subsection{Hirzebruch surfaces}\label{sec.hirz}
Recall the following.
\begin{definition}
The Hirzebruch surface $\mathbb{F}_e$ is the projective bundle $\mathbb{P}(\cO_{\mathbb{P}^1}(-e)\oplus\cO_{\mathbb{P}^1})$ on ${\mathbb{P}^1}$, where we take $e \geq 0$. Denote by $C$ the section~$(0:1)$ in $\mathbb{F}_e$.
\end{definition}

The fan of $\mathbb{F}_e$ is as in Figure~\ref{fig.hizfan} with rays $v_l$ for $l=1,\dots,4$~\cite[Example~4.1.8]{CLS}.
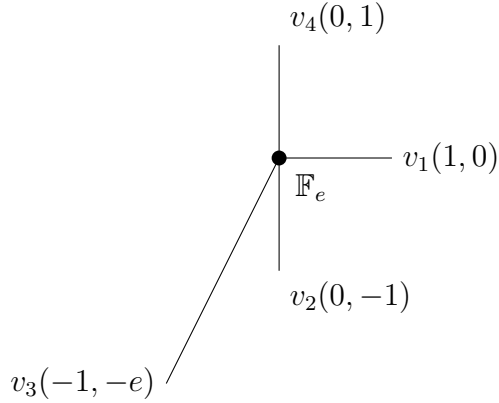
\begin{figure}[!h]
\centering
\begin{tikzpicture}[scale=1.5]
  \coordinate (V0) at (0,0);
  \coordinate (V1) at (1,0);
  \coordinate (V2) at (0,1);
  \coordinate (V3) at (0,-1);
  \coordinate (V4) at (-1,-2);

  \draw (V1) -- (V0); 
  \draw (V2) -- (V0); 
  \draw (V3) -- (V0); 
  \draw (V4) -- (V0); 

  \node[circle,fill,inner sep=1.5pt] at (V0) {};

  \node[right] at (V1) {$v_1(1,0)$};
  \node[above right] at (V2) {$v_4(0,1)$};
  \node[below right] at (V3) {$v_2(0,-1)$};
  \node[left]        at (V4) {$v_3(-1,-e)$};
\end{tikzpicture}  
\caption{The fan of $\mathbb{F}_e$}
\label{fig.hizfan}
\end{figure}

We complete the toric description of $\mathbb{F}_e$ and give intersection numbers, as follows.

\begin{proposition}\label{prop.curvesonhirzebruch}
We have:
\begin{itemize}
\item $V(v_2)=C$ with $C^2=-e$ in $\mathbb{F}_e$.
\item For $l=1,3$, $V(v_l)$ is a fibre of \,$\mathbb{F}_e$, satisfying $V(v_l)^2=0$ and $V(v_l)\cdot C=1$.
\end{itemize}
\end{proposition}
\begin{proof} This is standard. The fibre claim is because the bundle projection $\mathbb{F}_e\to{\mathbb{P}^1}$ is induced by 
the projection $N_{\mathbb{F}_e} \to N_{\PP^1}$ onto the first coordinate~\cite[Example~3.3.5 and Theorem~3.3.19]{CLS}.
\end{proof}

\subsection{The $A$-Hilbert scheme}\label{AHilb}
Let $A\subset \SL(3,\C)$ be a finite diagonal cyclic subgroup and write elements of $A$ as $\diag(\xi^{a},\xi^{b},\xi^{c})$ where $\xi$ is a primitive $r^{\text{th}}$ root of unity. The lattice for the fan $\Sigma_{\C^3/A}$ of $\C^3/A$ is given by an overlattice $L\supset\Z^3$ generated by 
\begin{equation*}
    \big\{\tfrac{1}{r}(a,b,c)\big|
    \diag(\xi^{a},\xi^{b},\xi^{c})
    \in A \text{ and }a,b,c\geq 0\big\}.
\end{equation*}

Write $\AHilb(\C^3)$ for the $A$-orbit Hilbert scheme of $\C^3$. It parametrizes $A$-invariant smoothable zero-dimensional subschemes of $\C^3$ of length $ |A|$~\cite{Nakamura}. For convenience, we put $X(A)=\AHilb(\C^3)$.\par

 By \cite{IR,Reid,CR,Craw}, the fan $\Sigma_{X(A)}$ of $X(A)$ is a refinement of $\Sigma_{\C^3/A}$ by inserting additional rays as follows.
 \begin{equation*}
    \Sigma_{X(A)}(1)=\Sigma_{\C^3/A}(1)\:\bigcup\:\big\{\tfrac{1}{r}(a,b,c)\big|
    \diag(\xi^{a},\xi^{b},\xi^{c})
    \in A,\ a+b+c=r\big\}
\end{equation*}
By \cite{Nakamura}, the projective toric resolution 
\begin{equation*}
    f\colon  X(A)\longrightarrow \C^3/A
\end{equation*}
given by this refinement is a projective crepant resolution, thence Calabi--Yau. By \cite[Proposition~11.1.10]{CLS}, the irreducible compact exceptional surfaces in $X(A)$ are given by $V(n)$ where $n\in \Sigma_{X(A)}(1)$ but $n\not\in \Sigma_{\C^3/A}(1)$. 

\section{Exceptional surfaces}\label{sec.cyc}
We now take $X=\AHilb(\C^3)$ with $A=\mu_7$ acting with weights~$(1,2,4)$. As discussed above, we can draw the toric fan $\Sigma$ of $X$ as in Figure~\ref{fig.fan}, and the irreducible compact exceptional surfaces are $V(\rho_l)$ for $l=1,2,3$. By examination, we have
\begin{equation}\label{equation.conerel}
\rho_4 = 2\rho_1 - \rho_2 \qquad
\rho_5 = 2\rho_2 - \rho_3 \qquad
\rho_6 = 2\rho_3 - \rho_1
\end{equation}
and so the overlattice $L$ is generated by $\rho_l$ for $l=1,2,3$.

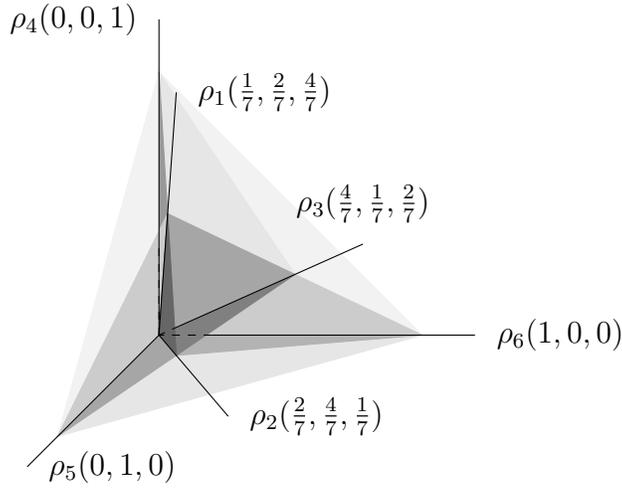
\begin{figure}[h!]
    \begin{center}
\begin{tikzpicture}[scale=3.5]
\coordinate (V1) at (1.2,0,0); 
\coordinate (V2) at (0,1.2,0); 
\coordinate (V3) at (0,0,1.3);  
\coordinate (O) at (0,0,0);
\coordinate (P1) at (1,0,0); 
\coordinate (P3) at (0,1,0); 
\coordinate (P2) at (0,0,1);
\coordinate (O1) at (1/7,4/7,2/7);
\coordinate (O2) at (2/7,1/7,4/7);
\coordinate (O3) at (4/7,2/7,1/7);
\coordinate (O11) at (2/7,8/7,4/7);
\coordinate (O22) at (8/7,4/7,16/7);
\coordinate (O33) at (6/7,3/7,1.5/7);
\coordinate (O333) at (0.6/7,0.3/7,0.15/7);
\coordinate (V33) at (0,0.33,0); 
\coordinate (V11) at (0.2,0,0); 

\filldraw[black!5] (O) -- (P3) -- (P1);
\filldraw[black!5] (O) -- (P3) -- (P2);
\filldraw[black!10] (O) -- (P1) -- (P2);
\filldraw[black!10] (O) -- (O3) -- (P3);
\filldraw[black!20] (O) -- (O3) -- (P1);
\filldraw[black!35] (O) -- (O3) -- (O1);
\filldraw[black!40] (O) -- (O2) -- (P3);
\filldraw[black!20] (O) -- (P2) -- (O1);
\filldraw[black!30] (O) -- (O2) -- (P1);
\filldraw[black!35] (O) -- (O2) -- (P2);
\filldraw[black!50] (O) -- (O2) -- (O3);
\filldraw[black!60] (O) -- (O1) -- (O2);
\draw (V1) -- (V11);
\draw (V2) -- (O);
\draw (V3) -- (O);
\draw[dashed] (V11) -- (O);
\draw[dashed] (V33) -- (O);
\draw[dashed] (O333) -- (O);
\draw (O11) -- (O);
\draw (O22) -- (O);
\draw (O33) -- (O333);

\node[label={0:$\rho_1(\tfrac{1}{7},\tfrac{2}{7},\tfrac{4}{7})$}] at (O11) {};
\node[label={0:$\rho_2(\tfrac{2}{7},\tfrac{4}{7},\tfrac{1}{7})$}] at (O22) {};
\node[label={90:$\rho_3(\tfrac{4}{7},\tfrac{1}{7},\tfrac{2}{7})$}] at (O33) {};
\node[label={180:$\rho_4(0,0,1)$}] at (V2) {};
\node[label={0:$\rho_6(1,0,0)$}] at (V1) {};
\node[label={0:$\rho_5(0,1,0)$}] at (V3) {};

\end{tikzpicture}
\end{center}
    
    \caption{The fan $\Sigma$ of $X$}
    \label{fig.fan}
\end{figure}

Consider the subfan $\Sigma_1$ of $\Sigma$ whose 3-dimensional cones meet~$\rho_1$, as in Figure~\ref{fig.subfan}. Then $U_1\coloneqq X_{\Sigma_1}$ is open in $X$ and $V(\rho_1)\subset U_1$.
\begin{figure}[tb]
    \centering
    \begin{minipage}[b]{0.48\textwidth}
        \centering
        \begin{tikzpicture}[scale=3]
\coordinate (V1) at (1.2,0,0); 
\coordinate (V2) at (0,1.2,0); 
\coordinate (V3) at (0,0,1.3);  
\coordinate (O) at (0,0,0);
\coordinate (P1) at (1,0,0); 
\coordinate (P3) at (0,1,0); 
\coordinate (P2) at (0,0,1);
\coordinate (O1) at (1/7,4/7,2/7);
\coordinate (O2) at (2/7,1/7,4/7);
\coordinate (O3) at (4/7,2/7,1/7);
\coordinate (O11) at (2/7,8/7,4/7);
\coordinate (O22) at (8/7,4/7,16/7);
\coordinate (O33) at (6/7,3/7,1.5/7);
\coordinate (O333) at (0.6/7,0.3/7,0.15/7);
\coordinate (V33) at (0,0.33,0); 
\coordinate (V11) at (0.2,0,0); 

\filldraw[black!5] (O) -- (P3) -- (P2);
\filldraw[black!10] (O) -- (O3) -- (P3);
\filldraw[black!35] (O) -- (O3) -- (O1);
\filldraw[black!40] (O) -- (O2) -- (P3);
\filldraw[black!20] (O) -- (P2) -- (O1);
\filldraw[black!35] (O) -- (O2) -- (P2);
\filldraw[black!50] (O) -- (O2) -- (O3);
\filldraw[black!60] (O) -- (O1) -- (O2);
\draw (V1) -- (V11);
\draw (V2) -- (O);
\draw (V3) -- (O);
\draw[dashed] (V11) -- (O);
\draw[dashed] (V33) -- (O);
\draw[dashed] (O333) -- (O);
\draw (O11) -- (O);
\draw (O22) -- (O);
\draw (O33) -- (O333);

\node[label={0:$\rho_1$}] at (O11) {};
\node[label={0:$\rho_2$}] at (O22) {};
\node[label={90:$\rho_3$}] at (O33) {};
\node[label={180:$\rho_4$}] at (V2) {};
\node[label={0:$\rho_5$}] at (V3) {};

\end{tikzpicture}
        \caption{The subfan $\Sigma_1$}
        \label{fig.subfan}
    \end{minipage}
    \hfill
    \begin{minipage}[b]{0.45\textwidth}
        \centering
        \begin{tikzpicture}[scale=1.5]
            \coordinate (V1) at (0,0);
            \coordinate (V2) at (0,-1);
            \coordinate (V3) at (1,0);
            \coordinate (V4) at (0,1);
            \coordinate (V5) at (-1,-2);

            \draw (V2) -- (V1); 
            \draw (V3) -- (V1); 
            \draw (V4) -- (V1); 
            \draw (V5) -- (V1); 

            \node[circle,fill,inner sep=1.5pt] (Hirz) at (V1) {};

  \node[below right] at (V2) {$\newrays_2(0,-1)$};
  \node[right] at (V3) {$\newrays_3(1,0)$};
  \node[above right]    at (V4) {$\newrays_4(0,1)$};
      \node[left] at (V5) {$\newrays_5(-1,-2)$};
            
        \end{tikzpicture}    
        \caption{The fan $\overline{\Sigma}$}
        \label{fig.hizfanplus}
    \end{minipage}
\end{figure}
Introduce a fan $\overline{\Sigma}$ with one-dimensional cones as follows, as  in Figure~\ref{fig.hizfanplus}.
\begin{equation*}
\newrays_2 = (0,-1) \qquad
\newrays_3 = (1,0) \qquad
\newrays_4 = (0,1) \qquad
\newrays_5 = (-1,-2) 
\end{equation*}
Recalling the description of the Hirzebruch surface from Section~\ref{sec.hirz}, and comparing Figure~\ref{fig.hizfanplus} with Figure~\ref{fig.hizfan} there, we have $
X_{\overline{\Sigma}} \cong \mathbb{F}_2$. Take then the Cartier divisor $D = -\sum_{l=2}^5 V(\newrays_l)$ on $\mathbb{F}_2$ and recall the fan $\overline{\Sigma} \times D$ given in Section~\ref{sect.linebundles}. Then $X_{\overline{\Sigma} \times D}$ is the total space of $\cO_{\mathbb{F}_2}(D)$.

\begin{proposition}\label{prop.hirzopen}
There is a lattice isomorphism $ L \cong N_{\overline{\Sigma} \times D}$ which induces an isomorphism of $U_1$ with the total space of $\cO_{\mathbb{F}_2}(D)$.
\end{proposition}

\begin{proof}
 By construction, the fan $\overline{\Sigma} \times D$ has one-dimensional cones as follows. 
\begin{equation*}
\rho'_1 = (0,0,1) \qquad
\rho'_2 = (0,-1,1) \qquad
\rho'_3 = (1,0,1) \qquad
\rho'_4 = (0,1,1) \qquad
\rho'_5 = (-1,-2,1)  \end{equation*}
The $\rho'_l$ for $l=1,2, 3$ generate the lattice $\mathbb{Z}^3$ for $\overline{\Sigma} \times D$. It is then easily checked that 
\begin{equation}
\label{equation.conerelimage}
\rho'_4 = 2\rho'_1 - \rho'_2 \qquad
\rho'_5 = 2\rho'_2 - \rho'_3
\end{equation}
and so, after comparing
\eqref{equation.conerel} and~\eqref{equation.conerelimage}, letting $\rho_l \mapsto \rho'_l$ gives the required lattice isomorphism. By inspection, this induces an isomorphism $
\Sigma_1 \cong \overline{\Sigma} \times D$, and the last part is then standard~\cite[Theorem~3.3.4]{CLS}.
\end{proof}

\begin{remark} The divisor $D$ is the torus-invariant canonical divisor of $\mathbb{F}_2$~\cite[Theorem~8.2.3]{CLS}.
\end{remark}

\begin{proposition}\label{prop.hirz}
    The surface ${S}_l\coloneqq V(\rho_l)\subset X$ for $l=1,2,3$ is isomorphic to $\mathbb{F}_2$.
\end{proposition}
\begin{proof}
By rotational symmetry, we only need to consider $l=1$. By Proposition~\ref{prop.divisor}, $V(\rho'_1)$ is the zero section of $\cO_{\mathbb{F}_2}(D)$, and the result follows via the isomorphism of Proposition~\ref{prop.hirzopen}.
\end{proof}

Write $C_l$ for the section~$(0:1)$ in ${S}_l$. Recall that $C_l\cong\mathbb{P}^1$.

\begin{proposition}\label{prop.2inter}
The curve $C_l$ is a fibre of ${S}_{l+1}$, and we have:
\begin{itemize}
\item $C_l^2=-2$ in ${S}_l$.
\item $C_l^2=0$ in ${S}_{l+1}$.
\end{itemize}
Moreover, the intersection $\bigcap_{\,l=1}^{\,3} {S}_l=\{p\}$ for a point $p\in X$.
\end{proposition}
\begin{proof}
By symmetry, we may prove the claim for $S_1$. Using~\cite[Theorem~3.2.6(d)]{CLS}
\[
S_1 \cap S_2
= V(\rho_1)\cap V(\rho_2)
= V(\langle\rho_1,\rho_2\rangle)
\cong V(\langle\rho'_1,\rho'_2\rangle)
= i\,V(\tau_2)
\]
where the isomorphism is via
Proposition~\ref{prop.hirzopen}, and the final equality follows from the embedding
\(i\colon S_1 \to U_1 \subset X\) in Proposition~\ref{prop.divisor} and the fact
that \(\tilde\tau_2=\langle\rho'_1,\rho'_2\rangle\) by the construction of
\(\overline{\Sigma}\times D\). Hence \(S_1 \cap S_2=C_1\) since \(V(\tau_2)\) is the section
 \((0\!:\!1)\) in \(X_{\overline{\Sigma}}\cong \F_2\) by Proposition~\ref{prop.curvesonhirzebruch}, and furthermore \(C_1^2=-2\) in~\(S_1\). By symmetry, \(S_l\cap S_{l+1}=C_l\) and so similarly to the above we find $C_3
= i\,V(\tau_3)$ so that $C_3$ is a fibre of $S_1$ and \(C_3^2=0\) in~\(S_1\) as required.

     For the last part, $\bigcap_{\,l=1}^{\,3} {S}_l=\bigcap_{\,l=1}^{\,3} V(\rho_l)=V(\sigma)$ where  $\sigma=\langle\rho_1,\rho_2,\rho_3\rangle$. This is $\{p\}$ for $p\in X$ because $L$ is generated by $\rho_1,\rho_2,\rho_3$.
 \end{proof}

\section{Proof}\label{sec.proof}

As in Section~\ref{sec.cyc}, we take $X=\AHilb(\C^3)$ where $A=\mu_7$ acts with weights~$(1,2,4)$. By Proposition~\ref{prop.hirz}, all three exceptional surfaces ${S}_l$ in $X$ are isomorphic to $\mathbb{F}_2$. Let $i_l\colon  {S}_l\to X$ be the closed embedding, and set $\mathcal{E}_l\coloneqq i_{l*}\mathcal{O}_{{S}_l}$.

\begin{proposition}
    Each $\mathcal{E}_l$ is spherical and induces an autoequivalence $T_l$ of $D(X)$.
\end{proposition}
\begin{proof}
     Since ${S}={S}_l$ is rational and connected, $\mathcal{O}_{{S}}$ is an exceptional object in $D({S})$. The variety $X$ is smooth, $i=i_l$ is the embedding of a complete connected hypersurface, and $i^* \omega_{X}$ is trivial because $X$ is Calabi--Yau. The claim then follows by \cite[Proposition~3.15 and remarks above]{ST} which notes that $X$ quasi-projective suffices. 
\end{proof}

\begin{proposition}\label{prop.2-chain}
    The $\{\mathcal{E}_l,\mathcal{E}_m\}$ are $A_2$-configurations for each pair $\{l,m\}\subset\{1,2,3\}$. In addition, $\dHomk_X(\cE_l,\cE_{l+1})\cong\C[-1]$ for $l=1,2,3$ with indices taken modulo $3$.
\end{proposition}
\begin{proof}
By rotational symmetry and Serre duality, which applies to the compactly-supported objects $\cE_l$, it is enough to show $\dHomk_X(\cE_1,\cE_2)\cong\C[-1]$.
    By Proposition~\ref{prop.extgen}
    \begin{equation*}
        \mathcal{H}om_{X}(\cE_1,\cE_2) \cong  k_{1*}  \cO_{C_{1}}({S}_1)[-1]
\end{equation*}
where $k_{1}$ denotes the embedding of $C_{1}$ in $X$. Now $\cO_{C_{1}}({S}_1)$ is trivial using Lemma~\ref{lemma.restocurve}, noting that $C_{1}^2=0$ in ${S}_2$ by Proposition~\ref{prop.2inter}. Taking global sections, we conclude.
\end{proof}

By construction, ${S}_2$ is a $\PP^1$-bundle over $C_{2}$. Write $C_{12}$ for the fibre over a point $q \in C_{2} \setminus \{p\}$, as shown in Figure~\ref{fig.rel}.

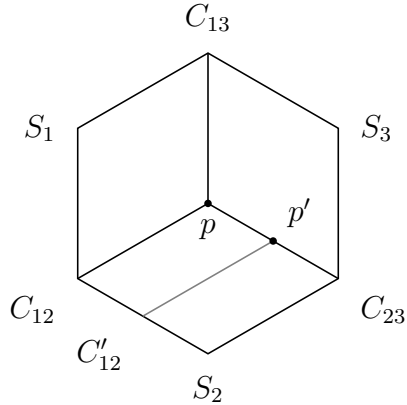
\begin{figure}[h!]
    \centering
    \begin{tikzpicture}[scale=2,semithick,
dot/.style = {circle, fill=black, inner sep=1pt}]
        \coordinate (O) at (0,0);
        \coordinate (V2) at (90:1); 
\coordinate (V3) at (210:1); 
\coordinate (V1) at (330:1);
\coordinate (S2) at (-90:1); 
\coordinate (S1) at (150:1); 
\coordinate (S3) at (30:1);
\coordinate (V21) at ($ (V3) !1/2! (S2) $);
\coordinate (V22) at ($ (O) !1/2! (V1) $);

\draw (S3) -- (V2) -- (S1) -- (V3) -- (S2) -- (V1) -- cycle; 
\draw (V1) -- (O);
\draw (V2) -- (O);
\draw (V3) -- (O);
\draw[black!50] (V21) -- (V22);

\node[label={180:${S}_1$}] at (S1) {};
\node[label={270:${S}_2$}] at (S2) {};
\node[label={0:${S}_3$}] at (S3) {};
\node[label={-30:$C_{2}$}] at (V1) {};
\node[label={210:$C_{1}$}] at (V3) {};
\node[label={230:$C_{12}$}] at (V21) {};
\node[label={90:$C_{3}$}] at (V2) {};
\node[dot,label={270:$p$}] at (O) {};
\node[dot,label={30:$q$}] at (V22) {};
    \end{tikzpicture}
    \caption{The fibre $C_{12}$ in ${S}_2$}
    \label{fig.rel}
\end{figure}

\begin{proposition}\label{prop.eq}
     $T_1\cE_2$ is isomorphic to the pushforward of $\cO_{{S}_1 \cup {S}_2}(C_{12})$ to $X$. 
\end{proposition}
\begin{proof}
There is a short exact sequence
\begin{equation*}
    0\longrightarrow \mathcal{I}_{C_{1}|{S}_2}\longrightarrow\mathcal{O}_{{S}_1\cup{S}_2}\longrightarrow\mathcal{O}_{{S}_1}\longrightarrow 0
\end{equation*}
on ${S}_1\cup{S}_2$, which is non-split as ${S}_1\cap{S}_2=C_{1}$ is non-empty. Now $\mathcal{I}_{C_{1}|{S}_2}\cong \mathcal{I}_{{C_{12}}|{S}_2}$ as $C_{1}$ and $C_{12}$ are linearly equivalent. Tensoring by $\cO_{{S}_1\cup{S}_2}(C_{12})$ we get the exact sequence
\begin{equation*}
    0\longrightarrow \mathcal{O}_{{S}_2}\longrightarrow\cO_{{S}_1\cup{S}_2}(C_{12})\longrightarrow\mathcal{O}_{{S}_1}\longrightarrow 0
\end{equation*}
after noting that ${S}_1\cap C_{12}=\emptyset$. Pushforward then gives a non-split short exact sequence on~$X$ which, using Proposition~\ref{prop.2-chain} and Lemma~\ref{lemma.exten}(2), proves the claim.   
\end{proof}

\begin{proposition}\label{prop.ext1}
$T_1\cE_2$ lies in the right orthogonal of $\cE_3$, i.e.~ $T_1\cE_2\in\cE_3^\perp$.
\end{proposition}
\begin{proof}
Write $i$ for the  embedding  of ${S}_1 \cup {S}_2$ in $X$. By Proposition~\ref{prop.eq}  we have 
    \begin{align*}
    \mathcal{H}om_X(\mathcal{E}_3, T_1\cE_2)
        &\cong\mathcal{H}om_X(i_{3*}\cO_{{S}_3},  i_{*}\cO_{{S}_1 \cup {S}_2}(C_{12}))
    \end{align*}
and by Proposition~\ref{prop.extgen} this is $k_{*} \cL[-1]$ where $\cL=\mathcal{O}_{C}({S}_3)\otimes \cO_{C}(C_{12})$ for $C = C_{2} \cup C_{3}$ and $k$ is the embedding of $C$ in $X$. It thence suffices to show that $\cL$ has no cohomology.

  We have
$  \cL|_{C_{3}}
    \cong \mathcal{O}_{C_{3}}({S}_3)
    $ because $C_{3}\cap C_{12}=\emptyset$. But this is trivial
using Lemma~\ref{lemma.restocurve}, noting that $C_{3}^2=0$ in ${S}_1$ by Proposition~\ref{prop.2inter}. Writing $\pi\colon C\to C_{2}$ for the contraction of the branch $C_{3}$ to~$p$, we thence have $\cL\cong \pi^*(\cL|_{C_{2}})$.

Recall that ${S}_2$ is a $\PP^1$-bundle over $C_{2}$ and $C_{12}$ is the fibre over $q\in C_{2}$,  hence
\begin{equation*}
    \cL|_{C_{2}}
    \cong \cO_{C_{2}}({S}_3)\otimes \cO_{C_{2}}(q)
    \cong \cO_{C_{2}}(-2)\otimes\cO_{C_{2}}(1)
    \cong \cO_{C_{2}}(-1)
\end{equation*}
where we again use Lemma~\ref{lemma.restocurve}, noting that $C_{2}^2=-2$ in ${S}_2$ by Proposition~\ref{prop.2inter}. This has no cohomology, so using the projection formula and $\pi_*\cO_{C}\cong\cO_{C_{2}}$ we are done.
\end{proof}

\begin{definition}[\hspace*{-0.45em}\cite{GM,Qiu}]
A group $G$ associated with the quiver with potential $(Q,W=cba)$ in Figure~\ref{fig.quiver} is defined as follows.
    \begin{equation*}
    G=\langle\, g_1,g_2,g_3 \,|\,                g_1g_2g_3g_1=g_2g_3g_1g_2,\,
                g_ig_jg_i=g_jg_ig_j\rangle
\end{equation*}
\end{definition}

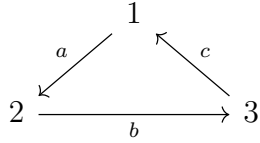
\begin{figure}[h!]
\begin{center}
\begin{tikzcd}
	& 1 \\
	2 && 3
	\arrow[from=1-2, to=2-1,"a"']
	\arrow[from=2-1, to=2-3,"b"']
	\arrow[from=2-3, to=1-2,"c"']
\end{tikzcd}
\end{center}

\caption{Quiver $(Q,W=cba)$}

\label{fig.quiver}
\end{figure}

\begin{theorem}\label{thm.main}
 The group $G$ above acts faithfully on $D(X)$ by $g_i \mapsto T_i$ for $i=1,2,3$.

\end{theorem}
\begin{proof}

    Using Proposition~\ref{prop.2-chain} and~\cite[Proposition~2.13]{ST}, $T_iT_jT_i\cong T_jT_iT_j$.
It remains to show $T_1T_2T_3T_1\cong T_2T_3T_1T_2$ in $\mathrm{Aut}\,D(X)$ or equivalently $T_2^{-1}T_1T_2\cong(T_3T_1)T_2(T_3T_1)^{-1}$.
Similarly to \cite[Proposition~2.13]{ST}, using that $T_{\Phi(\cE)} \cong \Phi T_{\cE} \Phi^{-1}$ for any $\Phi\in\mathrm{Aut}\,D(X)$, it is enough to show $T_2^{-1}\cE_1\cong T_3(T_1\cE_2)$.
By Proposition~\ref{prop.2-chain} and Lemma~\ref{lemma.exten}(1), we have $T_2^{-1}\cE_1\cong T_1\cE_2$.
Using Proposition~\ref{prop.ext1}, we have $T_3(T_1\cE_2)\cong T_1\cE_2$.
Combining gives the desired action.

For faithfulness, the assumptions of Proposition~\ref{prop.a23}(2) are satisfied by Propositions~\ref{prop.2-chain} and~\ref{prop.ext1}, and hence $\{\cE_1,\cE_2,T_2\cE_3\}$ is an $A_3$-configuration. (Note that Proposition~\ref{prop.a23} was given for $X$ projective, but quasi-projective suffices because Serre duality applies to the compactly-supported objects $\cE_l$.)
Then the action of the subgroup of $G$ generated by $\{g_1,g_2,g_2g_3g_2^{-1}\}$ is precisely the faithful action of $\Br_4$ on $D(X)$ from~\cite[Theorem~1.3]{ST}, since the conjugate $g_2g_3g_2^{-1}$ acts via $T_2T_3T_2^{-1} \cong T_{T_2\cE_3}$.
But $\{g_1,g_2,g_2g_3g_2^{-1}\}$ generates~$G$, and so $G\cong \Br_4$ and acts on $D(X)$ faithfully.
\end{proof}

%\newpage

\end{document}